\documentclass{article}

\PassOptionsToPackage{compress}{natbib}

\usepackage[final]{neurips_2021}

\usepackage[utf8]{inputenc} 
\usepackage[T1]{fontenc}    
\usepackage[colorlinks=true,citecolor=red]{hyperref}       
\hypersetup{
	colorlinks,
	citecolor=red}
\usepackage{url}            
\usepackage{booktabs}       
\usepackage{amsfonts}       
\usepackage{nicefrac}       
\usepackage{microtype}      
\usepackage{xcolor}         

\title{Accelerated Almost-Sure Convergence Rates for Nonconvex Stochastic Gradient Descent using Stochastic Learning Rates}

\author{Theodoros~Mamalis \\
	Department of Electrical and Computer Engineering\\
	University of Illinois at Urbana-Champaign\\
	306 N Wright St, Urbana, IL 61801, USA \\
	\texttt{mamalis2@illinois.edu} \\
	\And
	Du{\v s}an Stipanovi{\' c} \\
	Coordinated Science Laboratory\\
	University of Illinois at Urbana-Champaign\\
	1308 W Main St, Urbana, IL 61801, USA \\
	\texttt{dusan@illinois.edu} \\
	\AND
	Petros Voulgaris \\
	Department of Mechanical Engineering\\
	University of Nevada\\
	Reno, NV 89557, USA \\
	\texttt{pvoulgaris@unr.edu} \\
}

\usepackage{url}

\usepackage{caption}
\usepackage{subcaption}
\usepackage{graphicx}

\usepackage{amsmath}
\usepackage{amssymb}
\usepackage{amsthm}

\newtheorem{theorem}{Theorem}
\newtheorem{lemma}{Lemma}

\newtheorem{assumption}{Assumption}

\newtheorem{proposition}{Proposition}

\usepackage{algorithmicx}
\usepackage{algorithm}
\usepackage{algpseudocode}
\usepackage{etoolbox}

\makeatletter
\def\blfootnote{\gdef\@thefnmark{}\@footnotetext}
\makeatother

\begin{document}

\maketitle
\thispagestyle{empty}
\pagestyle{empty}


\blfootnote{This work is supported by the grant from the National Robotics Initiative grant titled “NRI: FND:COLLAB: Multi-Vehicle Systems for Collecting Shadow-Free Imagery in Precision Agriculture” (grant no.2019-04791/project accession no. 1020285) from the USDA National Institute of Food and Agriculture.
}

\begin{abstract}
Large-scale optimization problems require algorithms both effective and efficient. One such popular and proven algorithm is Stochastic Gradient Descent which uses first-order gradient information to solve these problems. This paper studies almost-sure convergence rates of the Stochastic Gradient Descent method when instead of deterministic, its learning rate becomes stochastic. In particular, its learning rate is equipped with a multiplicative stochasticity, producing a stochastic learning rate scheme. Theoretical results show accelerated almost-sure convergence rates of Stochastic Gradient Descent in a nonconvex setting when using an appropriate stochastic learning rate, compared to a deterministic-learning-rate scheme. The theoretical results are verified empirically.
\end{abstract}

\section{Introduction}

Solving large-scale optimization problems usually requires the optimization algorithm to process large amounts of data which can lead to long computational times as well as large memory requirements. For this reason, simple and cost-effective algorithms need to be used. The most prominent examples are the so-called gradient algorithms which utilize gradient information to iteratively estimate a minimizer for the optimization problem, with perhaps the most prominent algorithm being Stochastic Gradient Descent (SGD), initially introduced in \cite{SGD1951}. 

Even though in most optimization algorithms (including SGD) commonly the learning rate $ \eta_k $ (where $ k $ is the iteration count) also referred to as step size, is taken to be deterministic, the novelty in this work is that it introduces the problem where the learning rate becomes stochastic by equipping it with multiplicative stochasticity. In general, it has been observed that minimization performance can be noticeably improved under appropriate learning schemes  e.g., under deterministic adaptive learning rate schemes. Examples include ADAM \cite{ADAM15} and some variants (e.g., AMSGrad \cite{ADAMGRAD18} or ADAMW \cite{ADAMW2019}) or precursors  (e.g., \cite{Streeter10}, \cite{Duchi11}). This work demonstrates that performance can be significantly improved if instead of just adaptive, the learning rate becomes both adaptive and stochastic.

In specific, the learning rate instead of $ \eta_k $ now becomes $ \eta_k u_k $  (where e.g., $\eta_k = \eta/\sqrt{k}$ or $\eta_k=\eta $, with $\eta$ a positive constant; with some abuse of terminology, $ \eta_k $ will be referred to as step size) and where $ u_k $ is a random variable, referred to as stochasticity factor (SF) in the rest of the paper. Because of the multiplicative nature of the SF 
the stochastic learning rate examined in this work will be referred to as Multiplicative-Stochastic-Learning-Rate (MSLR).
A stochastic learning rate whose SF is a uniformly distributed random variable at each iteration will be referred to as Uniform-Multiplicative-Stochastic-Learning-Rate (UMSLR). This work uses MSLR on SGD to prove accelerated almost sure convergence rates in a nonconvex and smooth setting compared to a deterministic learning rate.

The SGD algorithm has been a traditional topic of research investigated in numerous works, e.g., \cite{bertsekastsitsiklis2000,mertikopoulos2017}, \cite{loizou2018}. In \cite{mertikopoulos2020} it is shown that the gradient norm of SGD converges almost surely (a.s.). \cite{loizou2020} show convergence without knowledge of the smoothness constant using Line-Search \cite{NocedalWright2006} and Polyak stepsizes \cite{Polyak1987}. Almost sure convergence rates have been provided in works such as \cite{bottou2003, nguyen2018}. The authors in \cite{SHB_AS2021} derive almost sure convergence rates for the minimum squared gradient norm of SGD in a noncovex and smooth setting by utilizing a convergence result from \cite{ROBBINS1971}. In-expectation convergence analyses have been made in works such as \cite{nemirovski2009, bachmoulines2011,Bottou2018} and references therein. In-expectation convergence analysis of SGD for stochastic-learning-rates was made in \cite{mamalis2020exp_and_online_conv}.

In summary, this work provides accelerated a.s. convergence rates for SGD in the nonconvex and smooth setting using MSLR, compared to the deterministic learning rate case. Experiments
demonstrate the improved a.s. convergence rates empirically. In more detail, the main contributions of this paper are:
\begin{itemize}
	\item The introduction of the notion of stochastic learning rate schemes. Note that in this work a stochastic learning rate scheme should not depend on the stochasticity induced by the samples and how they are processed 
	(e.g., as is usually the case with the computation of the gradients in most of the stochastic gradient methods). 
	The stochasticity in the learning rates should be able to be directly controlled to be of any distribution, a requirement whose importance is discussed in the next point. It should be noted that this is in contrast to the stochasticity induced by randomly sampling the dataset which can be unable or more difficult to possess a (discrete) distribution with any prespecified property, regardless of the sampling patterns.
	\item Accelerated a.s. convergence rates for the SGD algorithm.  The introduction of stochastic learning rate schemes can be rigorously shown to provide better convergence rates than the deterministic learning rate versions of the SGD algorithm in the nonconvex and smooth settings for the latter. In specific, moments of the distribution of the SF enter the a.s. convergence rate of SGD, directly affecting its rates. This means that by choosing the distribution of the learning rate stochasticity, i.e., of the SF, to possess appropriate prespecified properties, the a.s. convergence rates of the resulting stochastic learning rate algorithm can be improved compared to its deterministic counterpart. 
	In absence of an SF the a.s. convergence rates reduce to those of the deterministic learning rate algorithms as expected. 
	\item Empirically demonstrating that MSLR schemes, and in specific UMSLR, exhibit significantly improved optimization performance for SGD in accordance to the theoretical results, for some popular datasets.
\end{itemize}

The remainder of the paper is organized as follows. Section \ref{sec:Prelim_and_Assumptions} states the context of this work in mathematical terms, along with a lemma and necessary assumption. The mathematical derivation of the almost sure convergence rates of the SGD algorithm with MSLR is given in Section \ref{sec:sgd_a_s_convergence_rates}. Section \ref{sec:Discussion on the Stochasticity Factor} constructs an appropriate SF which provides accelerated a.s. convergence rates for SGD, and provides a discussion on selecting hyperparameters that complement the SF distribution in accelerating performance. In Section \ref{sec:results}, the experimental results for SGD using a UMSLR scheme are presented and compared to the deterministic-learning rate case. Section \ref{sec:conclusion} concludes the paper with avenues of future work.

\section{Problem Formulation and Assumptions}
\label{sec:Prelim_and_Assumptions}
Let $ v $ be a random variable with distribution $ D $, $ X $ a set in $ \mathbb{R}^d $, and $ f_v(x) $ a function that depends on $ v $ and $ x \in X $. Then the optimization problem is:
\begin{equation}\label{eq:minimization}
\mathop { \min }\limits_{x  \in {\mathbb{R}^d}} \mathbb{E}_{v\sim D} [f_{v}(x)],
\end{equation}
Define $ f(x) := \mathbb{E}_{v\sim D} [f_{v}(x)] $. In the well-known case of Empirical Risk Minimization (ERM) $ \mathbb{E}_{v \sim D}[f_v(x)]=  \frac{1}{n}\sum\nolimits_{i = 1}^n {f_{v_i}({x})} $,  where $ v$ represents a random sample from the available training data and $ x $ the parameters to be learned. Assume $ f $ is smooth with $ L > 0 $ being its smoothness constant. It is assumed that at least one minimizer $ x_* $ to (\ref{eq:minimization}) exists yielding $ f_* $. 
Then, the SGD algorithm is given by:
\begin{equation}
\label{eq:SGD}
{x_{k + 1}} = {x_k} - {\eta_k}u_k\nabla {f_{{v_k}}}({x_k}).
\end{equation}	
 The stochastic learning rate is $  \eta_k u_k $ where $ u_k $ denotes the SF,
and $ \eta_k$ is deterministic and will be referred to as stepsize. This learning rate scheme will be referred to as MSLR.
Typical assumptions made in the context of stochastic approximations include either bounded gradients or bounded variance of the gradients. Below, a more general assumption is given:
\begin{assumption} [Expected Smoothness] \label{as:exp_smooth}
	There exist constants A,B,C s.t. for all $ x \in \mathbb{R}^d $:
	\begin{equation}
	\mathbb{E}_v[\left\| \nabla f_v(x) \right\|^2] \le A(f(x)-f_*) + B\left\| \nabla f(x) \right\|^2 + C.
	\end{equation}
\end{assumption}
This assumption is introduced in \cite{khaled2020} where is called Expected Smoothness and wherein the properties and significance of this assumption, especially for settings such as ERM, are discussed. Next, a lemma is presented:
\begin{lemma} \label{lemma:2.1}
	Consider a filtration $ (F_k)_k $, the nonnegative sequences of $ (F_k)_k $-adapted processes $ (V_k)_k $, $ (U_k)_k $ and $ (Z_k)_k $, and a sequence of positive numbers $(\gamma_k)_k$ such that  $ \sum\nolimits_k^{} Z_k < \infty$ almost surely and $ \prod\nolimits_k^{} (1 + \gamma_k) < \infty $ , and:
	\begin{equation}
	\forall k \in \mathbb{N},\,\,\mathbb{E}[{V_{k + 1}}|{F_k}] + {U_{k + 1}} \le (1 + {\gamma _k}){V_k} + {Z_k}.
	\end{equation}
	Then $ ({V_k})_k $ converges and $ \sum\nolimits_k^{} U_k < \infty $ almost surely.
\end{lemma}
This lemma (\cite{ROBBINS1971}) is used extensively in
the proofs of the theorems of this paper.
The following assumption includes
standard conditions for the step size that appear in the stochastic approximations literature and which usually hold in practice.
\begin{assumption} \label{cond:condition_for_stepsizes}
	The sequence $ {({\eta _k})_k} $ is decreasing, $ \sum\nolimits_k^{} {{\eta _k}}  = \infty$ , $\sum\nolimits_k^{} {\eta _k^2}  < \infty $ and $ \sum\nolimits_k^{} {\frac{{{\eta _k}}}{{\sum\nolimits_{j=0}^{k-1} {{\eta _j}} }}}  = \infty $.
\end{assumption}
Moreover, for the expected value of the SF the following assumption is made.
\begin{assumption} \label{cond:condition_for_stochasticfactor}
	The sequence $ \mathbb{E}_u[u_k] $ is monotone.
\end{assumption}
At this point it is noted that distributions for which (\ref{cond:condition_for_stochasticfactor}) holds are able to be constructed. Discussion on the distribution constructed in this work is deferred to Section \ref{sec:On_the_SF}.
\section{Accelerated Nonconvex SGD Almost-Sure Convergence Rates}
\label{sec:sgd_a_s_convergence_rates}
This sections presents results for accelerated a.s. convergence rates of the SGD algorithm in the nonconvex case. The first result is formulated as follows:
\begin{theorem} \label{thm:3.2}
	Consider the iterates of (\ref{eq:SGD}). Assume that Assumption \ref{as:exp_smooth} holds. Assume that the stochasticity factor is bounded, i.e., $ 0 < u_k \le c_2 $ with $ c_2:=\sup_k u_k $. 
	Assume that the first moment of the stochasticity factor is bounded, i.e., $ 0<\mu_1 \le \mathbb{E}_u[u_k] $  with $ \mu_1:=\inf_k \mathbb{E}_u[u_k] $. 
	Choose stochasticity factor $ u_k $ which satisfies \ref{cond:condition_for_stochasticfactor}. Assume that the stepsizes verify \ref{cond:condition_for_stepsizes}. 
	
	\textit{1.1a.} If $ \mathbb{E}_u[u_k] $  is decreasing,
	if $ Var_u[u_k] $ is increasing, if $ \mathbb{E}_u[u_k] >  Var_u[u_k] + 1 $,
	 and if $ 0 \le \eta_{k} \le 1/(BL\mathbb{E}_u[u_k]) $ for all $ k \in \mathbb{N} $, then:
	\begin{equation}
	\mathop {\min }\limits_{t = 0, \ldots ,k - 1} {\left\| {\nabla f({x_t})} \right\|^2} = o\left( {\frac{(\mathbb{E}_u[u_{k}]-Var_u[u_k])^{-1}}{{\sum\nolimits_{t = 0}^{k - 1} {{\eta _t}} }}} \right) \quad a.s.
	\end{equation}
	\textit{1.1b.} If $ \mathbb{E}_u[u_k] $ is decreasing,
	 and if $ 0 \le \eta_{k} \le 1/(BLc_2) $ for all $ k \in \mathbb{N} $, then:
	\begin{equation}
	\mathop {\min }\limits_{t = 0, \ldots ,k - 1} {\left\| {\nabla f({x_t})} \right\|^2} = o\left( {\frac{1}{\mathbb{E}_u[u_{k}]{\sum\nolimits_{t = 0}^{k - 1} {{\eta _t}} }}} \right) \quad a.s.
	\end{equation}
	\textit{1.2.} If $ \mathbb{E}_u[u_k] $ is increasing, 
	if $ Var_u[u_k] $ is decreasing, if $ \mathbb{E}_u[u_k]<1 $, if $ \frac{Var_u[u_k]}{\mathbb{E}_u[u_k]}<1 $, and 
if $ 0 \le \eta_{k} \le 1/(BL) $ for all $ k \in \mathbb{N} $, then:	
	\begin{equation}
	\mathop {\min }\limits_{t = 0, \ldots ,k - 1} {\left\| {\nabla f({x_t})} \right\|^2} = o\left( \frac{\mathbb{E}_u[u_{k}]-Var_u[u_{k}] }{{\sum\nolimits_{t = 0}^{k - 1} {{\eta _t}} }} \right) \quad a.s.
	\end{equation}
\end{theorem}
Several remarks on the SGD convergence rate results in the nonconvex case are in order. First, the assumptions of Theorem \ref{thm:3.2} guarantee the a.s. convergence rates presented in the theorem. However, whether the MSLR a.s. convergence rates are accelerated or not compared to the deterministic-learning-rate case, whose rate is $ o\left( (\sum\nolimits_{t = 0}^{k - 1} {{\eta _t}})^{-1} \right) $ (\cite{SHB_AS2021}),  depends on the choice of the SF distribution of MSLR SGD. 
To that end, it can be readily observed what properties the SF needs to satisfy so that the MSLR scheme produces faster a.s. convergence rates than the deterministic-learning-rate case.
Firstly, for 1.1a, for the moments of the SF it should be that $ \mathbb{E}_u[u_{k}]>Var_u[u_{k}] + 1 $. Secondly, for  1.1b, it should be that $ \mathbb{E}_u[u_k] > 1 $ and thirdly, for 1.2, it should be that $ \mathbb{E}_u[u_{k}]<Var_u[u_{k}] + 1 $. However, the last condition can be equivalently written as  $0<\mathbb{E}_u[u_k](1-\frac{Var_u[u_k]}{\mathbb{E}_u[u_k]})<1  $. This is implied by the two conditions  $ \mathbb{E}_u[u_k]<1 $ and $ \frac{Var_u[u_k]}{\mathbb{E}_u[u_k]}<1 $ (for 1.2, these two acceleration conditions happen to also be theorem assumptions, but this is not true in general), given that $ Var_u[u_k] \ge 0 $. The reason that the SF is taken to satisfy these two conditions instead of  $ \mathbb{E}_u[u_{k}]<Var_u[u_{k}] + 1 $, which is a weaker condition, is because in practice they can be more easily checked for adaptive distributions, e.g., as in the case of a uniformly distributed SF discussed in Section \ref{sec:Discussion on the Stochasticity Factor} in more detail.

Second, the requirements of Theorem \ref{thm:3.2} are sufficient but not necessary, i.e., they can be replaced by weaker assumptions and Theorem \ref{thm:3.2} will still hold. Nonetheless, the weaker assumptions are not as easily verified theoretically, and perhaps experimentally, as the assumptions currently given. For example, the mean $ \mathbb{E}_u[u_k] $ and variance $ Var_u[u_k] $  monotonicity assumptions along with the assumptions that follow them in 1.1a and 1.2 could be replaced by the assumptions $  Var_u[u_{k+1}] - Var_u[u_k] > \mathbb{E}_u[u_{k+1}] - \mathbb{E}_u[u_k] $ and $  Var_u[u_{k+1}] - Var_u[u_k] < \mathbb{E}_u[u_{k+1}] - \mathbb{E}_u[u_k] $ respectively. These are satisfied when respectively $ Var'_u[u_k] > E'_u[u_k]  $ and $ Var'_u[u_k] > E'_u[u_k] $ hold, where the derivatives are with respect to $ k $. This is since if, e.g., $ Var'_u[u_k] > E'_u[u_k]  $ then  $ \frac{Var_u[u_{k+1}] - Var_u[u_k]}{k+1-k} > \frac{\mathbb{E}_u[u_{k+1}] - \mathbb{E}_u[u_k]}{k+1-k} $ which gives the required weaker assumption for 1.1a (respectively for 1.2 by replacing $ > $ with $ < $ in the previous). 
However, given that checking assumptions which include both the derivatives of the mean and variance of some adaptive distributions
can become complex quickly, the current assumptions that consider the monotonicity of each of the mean and variance separately, are deemed simpler to check than the assumptions dealing with a combination of the derivatives of these moments.

Third, it is noted that the various requirements that appear in Theorem \ref{thm:3.2} are consistent. In 1.1a, $ Var_u[u_k] $ cannot escape to infinity since it is less than a decreasing and positive $ \mathbb{E}_u[u_k] $. In 1.1b, $ \mathbb{E}_u[u_k] $ cannot diverge to negative infinity since it is positive, and in 1.2, $ \mathbb{E}_u[u_k] $ cannot diverge to infinity since it is less than unity. Moreover, in the denominator of 1.1a the difference $ \mathbb{E}_u[u_k]-Var_u[u_k] $ appears whereas in 1.1b the larger quantity $ \mathbb{E}_u[u_k] $ appears. This should not be taken that 1.1b provides faster convergence rates than 1.1a since in case 1.1a the stepsize, and therefore the term $ {{\sum\nolimits_{t = 0}^{k - 1} {{\eta _t}} }} $, is larger than in case 1.1b so there is a trade-off between the acceleration provided by 1.1a and 1.1b (e.g., if the stepsizes for cases 1.1a and 1.1b are $ \bar\eta_{k}=\frac{1}{BL\mathbb{E}_u[u_k]} $ and  \underbar{$\eta$}$_k=\frac{1}{BLc_2} $ respectively, and, assuming a common SF $ u_k $ for both cases, it is that
$ 1-\frac{Var_u[u_k]}{\mathbb{E}_u[u_k]} \ge \frac{1}{c_2} $, then the former gives faster convergence rates since then $ ({(\mathbb{E}_u[u_{k}]-Var_u[u_k]){\sum\nolimits_{t = 0}^{k - 1} {{\bar\eta _t}} }})^{-1} \ge $($\mathbb{E}_u[u_{k}]\sum\nolimits_{t = 0}^{k - 1}$ \underbar{$\eta$} $_t )^{-1} )$.


Finally, when the SF becomes constant unity, the deterministic-learning-rate a.s. convergence rate  is recovered by the MSLR SGD a.s. convergence rates in Theorem \ref{thm:3.2} since for $ u_k=1 $ it is $ \mathbb{E}_u[u_k]=1 $ and $ Var_u[u_k]=0 $ yielding $ o\left(  (\sum\nolimits_{t = 0}^{k - 1} {{\eta _t}})^{-1} \right) $.

All in all, Theorem \ref{thm:3.2} demonstrates that  for appropriate choices of SFs the MSLR scheme accelerates the a.s. convergence rates of SGD compared to the a.s. convergence rates of its deterministic-learning-rate counterpart.
\begin{proof}
	The proof follows the proof of Lemma 2 in \cite{khaled2020}.
	Starting with (\ref{eq:SGD}), equation (46) in \cite{khaled2020} becomes:
	\begin{equation}
\begin{aligned}
f(x_{k+1})  &\le f(x_k) + \left\langle {\nabla f(x_k), x_{k+1}-x_k} \right\rangle + \frac{L}{2}\left\| x_{k+1}-x_k \right\|^2 \\
& = f(x_k) - u_k \eta_k \left\langle {\nabla f(x_k), g(x_k)} \right\rangle  + \frac{L u_k^2 \eta_k^2}{2}\left\| g(x_k) \right\|^2
\end{aligned}
\end{equation}
	which ultimately results in:
	\begin{equation} \label{eq:following the steps in that paper:}
	\begin{aligned}
	&\mathbb{E}_k[f(x_{k+1} ) - f_*] + \frac{ \eta_{k}}{2} (2u_k-u^2_kLB\eta_k) \left\| \nabla f (x_k) \right\|^2  \\&\le (1+u_k^2 \eta_{k}^2 AL)(f(x_k)-f_*) + \frac{u_k^2\eta_k^2LC}{2}.
	\end{aligned}
	\end{equation}
	
	\subsection{Case 1a:	$ \mathbb{E}_u[u_{k+1}] \le \mathbb{E}_u[u_k] $ and $ \eta_k \le \frac{1}{LB\mathbb{E}_u[u_k]} $.} \label{case:1a}
	
	Using $ \eta_{k} \le 1 /(BL\mathbb{E}_u[u_k]) $ in (\ref{eq:following the steps in that paper:}):
\begin{equation} \label{eq:eq_another_for_case1.2}
\begin{aligned}
&\mathbb{E}_k[f(x_{k+1}) - f_*] + \frac{\eta_{k}}{2} (2u_k-\frac{u^2_k}{\mathbb{E}_u[u_k]}) \left\| \nabla f (x_k) \right\|^2  \le (1+u_k^2 \eta_{k}^2 AL)(f(x_k)-f_*) + \frac{u_k^2\eta_k^2LC}{2}.
\end{aligned}
\end{equation}	
From $ u_k\le c_2 $ and Assumption \ref{cond:condition_for_stepsizes} it is that $ \sum\nolimits_{k = 0}^\infty  \frac{{u_k^2\eta _k^2LC}}{2} <  \sum\nolimits_{k = 0}^\infty  \frac{{c_2^2\eta _k^2LC}}{2}  < \infty  $. Also it is that $ \sum\nolimits_{k = 0}^\infty  {\eta _k^2} < \infty  $, which means $ \prod\nolimits_{k = 0}^\infty  {(1 + {\mathbb{E}_u}[u_k^2]\eta _k^2AL)} <  \prod\nolimits_{k = 0}^\infty  {(1 + c_2^2\eta _k^2AL)} < \infty  $. Thus, from Lemma \ref{lemma:2.1}, $ f(x_k) - f_* $ converges a.s. Taking expectations with respect to $ u $ yields:
\begin{equation}
\begin{aligned}
&\mathbb{E}_k[f(x_{k+1} ) - f_*] + \frac{ \eta_{k}}{2} (2\mathbb{E}_u[u_k]- \frac{\mathbb{E}_u[u^2_k]}{\mathbb{E}_u[u_k]}) \left\| \nabla f (x_k) \right\|^2 \\&\le (1+\mathbb{E}_u[u_k^2] \eta_{k}^2 AL)(f(x_k)-f_*) + \frac{\mathbb{E}_u[u_k^2]\eta_k^2LC}{2}.
\end{aligned}
\end{equation}
Using $ \mathbb{E}_u[u^2_k] = Var_u[u_k] + \mathbb{E}_u[u_k]^2 $:
\begin{equation}
\begin{aligned}
&\mathbb{E}_k[f(x_{k+1} ) - f_*] + \mathbb{E}_u[u_k]\frac{ \eta_{k}}{2} (2-\frac{\mathbb{E}_u^2[u_k]+Var_u[u_k]}{\mathbb{E}^2_u[u_k]}) \left\| \nabla f (x_k) \right\|^2 \\&\le (1+\mathbb{E}_u[u_k^2] \eta_{k}^2 AL)(f(x_k)-f_*) + \frac{\mathbb{E}_u[u_k^2]\eta_k^2LC}{2}.
\end{aligned}
\end{equation}
This means:
\begin{equation}
\begin{aligned}
&\mathbb{E}_k[f(x_{k+1} ) - f_*] + \mathbb{E}_u[u_k]\frac{ \eta_{k}}{2} (1-\frac{Var_u[u_k]}{\mathbb{E}^2_u[u_k]}) \left\| \nabla f (x_k) \right\|^2 \\& \le (1+\mathbb{E}_u[u_k^2] \eta_{k}^2 AL)(f(x_k)-f_*)+ \frac{\mathbb{E}_u[u_k^2]\eta_k^2LC}{2}.
\end{aligned}
\end{equation}
Since $ \mathbb{E}_u[u_k]>Var_u[u_k]+1 $ from the assumptions of the theorem, it is that $ \mathbb{E}_u[u_k] \ge 1 $, which yields $ \mathbb{E}^2_u[u_k] \ge  \mathbb{E}_u[u_k] $ or, equivalently, $ -\frac{Var_u[u_k]}{\mathbb{E}^2_u[u_k]} \ge -\frac{Var_u[u_k]}{\mathbb{E}_u[u_k]}  $, since $ Var_u[u_k]\ge0 $. Thus:
\begin{equation}
\label{eq:eq_beforeuse_of_gk_2}
\begin{aligned}
&\mathbb{E}_k[f(x_{k+1} ) - f_*] + \mathbb{E}_u[u_k]\frac{ \eta_{k}}{2} (1-\frac{Var_u[u_k]}{\mathbb{E}_u[u_k]}) \left\| \nabla f (x_k) \right\|^2 \\& \le (1+\mathbb{E}_u[u_k^2] \eta_{k}^2 AL)(f(x_k)-f_*)+ \frac{\mathbb{E}_u[u_k^2]\eta_k^2LC}{2}.
\end{aligned}
\end{equation}
	Then, let:
	\begin{equation}\label{eq:recurrence_nonconvex_2}
	\begin{aligned}
		&{w_k} = \frac{{2{\eta _k}}}{{\sum\nolimits_{t = 0}^{k } {{\eta _t}} }}, \quad {g_0} = {\left\| {\nabla f({x_0})} \right\|^2}, \\&  {g_{k + 1}} = (1 - {w_k}){g_k} + {w_k}{\left\| {\nabla f({x_k})} \right\|^2}
	\end{aligned}
	\end{equation}
	for all $ k \in \mathbb{N} $. Moreover, since $ g_{k} $ is a linear combination of $ {\left\| {\nabla f({x_t})} \right\|^2}, t = 0, \cdots , k - 1$  it is that:
\begin{equation} \label{eq:g_k=grad_case1a}
{g_k} = \sum\nolimits_{t = 0}^{k - 1} {{\tilde{w}_t}} {\left\| {\nabla f({x_t})} \right\|^2}
\end{equation} 
for some sequence $ {({\tilde{w}_t})_t} $ with  $ \sum\nolimits_{t = 0}^{k - 1} {{\tilde{w}_t}}  = 1 $, and where $ 0 \le \tilde{w}_t \le 1 $ since $ \eta_t $ is decreasing.
	Then, using (\ref{eq:recurrence_nonconvex_2}) to replace $ {\left\| {\nabla f({x_k})} \right\|^2} $ in (\ref{eq:eq_beforeuse_of_gk_2}) yields:
	\begin{equation} 
\begin{aligned} \label{eq:eq_for_case1.2}
& \mathbb{E}_k[f(x_{k+1}) - f_*] + \mathbb{E}_u[u_{k }](1-\frac{Var_u[u_{k }]}{\mathbb{E}_u[u_{k }]} ) \frac{\sum\nolimits_{t=0}^{k} {\eta_t}}{2} g_{k+1} + \frac{\eta_k}{2} \mathbb{E}_u[u_k](1-\frac{Var_u[u_k]}{\mathbb{E}_u[u_k]}) g_k  \\ &\le (1+\mathbb{E}_u[u_k^2] \eta_{k}^2 AL)(f(x_k)-f_*)
+  \mathbb{E}_u[u_k](1-\frac{Var_u[u_k]}{\mathbb{E}_u[u_k]}) \frac{\sum\nolimits_{t=0}^{k-1} {\eta_t}}{2} g_{k} + \frac{\mathbb{E}_u[u_k^2]\eta_k^2LC}{2}.
\end{aligned}
\end{equation}
	Since $ \mathbb{E}_u[u_k] \ge \mathbb{E}_u[u_{k+1}] $ and $ Var_u[u_{k }] \le Var_u[u_{k+1 }] $:
	\begin{equation}
\begin{aligned}
& \mathbb{E}_k[f(x_{k+1}) - f_*] + \mathbb{E}_u[u_{k+1 }](1-\frac{Var_u[u_{k+1 }]}{\mathbb{E}_u[u_{k+1 }]} ) \frac{\sum\nolimits_{t=0}^{k} {\eta_t}}{2} g_{k+1} + \frac{\eta_k}{2} \mathbb{E}_u[u_k](1-\frac{Var_u[u_k]}{\mathbb{E}_u[u_k]}) g_k  \\ &\le (1+\mathbb{E}_u[u_k^2] \eta_{k}^2 AL)(f(x_k)-f_*)
+  \mathbb{E}_u[u_k](1-\frac{Var_u[u_k]}{\mathbb{E}_u[u_k]}) \frac{\sum\nolimits_{t=0}^{k-1} {\eta_t}}{2} g_{k} + \frac{\mathbb{E}_u[u_k^2]\eta_k^2LC}{2} 
\end{aligned}
\end{equation}
	and since $(1+\mathbb{E}_u[u_k^2] \eta_{k}^2 AL) > 1 $:
	\begin{equation}
\begin{aligned}
& \mathbb{E}_k[f(x_{k+1}) - f_*] + \mathbb{E}_u[u_{k+1 }](1-\frac{Var_u[u_{k+1 }]}{\mathbb{E}_u[u_{k+1 }]} ) \frac{\sum\nolimits_{t=0}^{k} {\eta_t}}{2} g_{k+1} + \frac{\eta_k}{2} \mathbb{E}_u[u_k](1-\frac{Var_u[u_k]}{\mathbb{E}_u[u_k]}) g_k  \\ &\le (1+\mathbb{E}_u[u_k^2] \eta_{k}^2 AL)(f(x_k)-f_*) +  (1+\mathbb{E}_u[u_k^2] \eta_{k}^2 AL)\mathbb{E}_u[u_k] \frac{\sum\nolimits_{t=0}^{k-1} {\eta_t}}{2} g_{k} + \frac{\mathbb{E}_u[u_k^2]\eta_k^2LC}{2}.
\end{aligned}
\end{equation}
	Or equivalently:
	\begin{equation}
\begin{aligned}
& \mathbb{E}_k[(f(x_{k+1}) - f_*)+\varphi_1(k+1) ] + \frac{\eta_k}{2} \mathbb{E}_u[u_k](1-\frac{Var_u[u_k]}{\mathbb{E}_u[u_k]}) g_k  \\ &\le
(1+\mathbb{E}_u[u_k^2] \eta_{k}^2 AL) ((f(x_k)-f_*) + \varphi_1(k))  + \frac{\mathbb{E}_u[u_k^2]\eta_k^2LC}{2}
\end{aligned}
\end{equation}
	where $ \varphi_1(k) :=\mathbb{E}_u[u_k](1-\frac{Var_u[u_k]}{\mathbb{E}_u[u_k]}) \frac{\sum\nolimits_{t=0}^{k-1} {\eta_t}}{2} g_{k} $.
	Then, from $ u_k \le c_2 $ which gives $ \mathbb{E}_u[u^2_k] \le c^2_2 $, and from Assumption \ref{cond:condition_for_stepsizes} it is that $ \sum\nolimits_{k = 0}^\infty  {{\mathbb{E}_u}[u_k^2]\eta _k^2AL}  < {c_2^2}AL\sum\nolimits_{k = 0}^\infty  {\eta _k^2}  < \infty  $, which means $ \prod\nolimits_{k = 0}^\infty  {(1 + {\mathbb{E}_u}[u_k^2]\eta _k^2AL)}  < \infty  $. Moreover, it is that  $\frac{c_2^2LC}{2}\sum\nolimits_{k = 0}^\infty \eta_k^2 < \infty $ from Assumption \ref{cond:condition_for_stepsizes}. Then, from Lemma \ref{lemma:2.1} this means that $  \sum\nolimits_{t = 0}^{\infty} {{\eta _t}}\mathbb{E}_u[u_t](1-\frac{Var_u[u_t]}{\mathbb{E}_u[u_t]})g_t < \infty  $ a.s. and also that:
		$$  \left( (f(x_{k}) - f_*) +  \varphi_1(k)  \right) \text{converges a.s. }$$
	This gives that  $ (\varphi_1(k))_k $ converges a.s. since it is that  
	$ (f(x_{k+1}) - f_*) $ converges a.s. Then, algebraically manipulating  $ \mathop {\lim }\limits_{k \to \infty } \varphi_1(k)  $ it is that:
	\begin{equation}
\begin{aligned}
\mathop {\lim }\limits_{k \to \infty } \frac{{{\eta _k}}}{{\sum\nolimits_{t = 0}^{k - 1} {{\eta _t}} }}\mathbb{E}_u[u_{k}](1-\frac{Var_u[u_{k}]}{\mathbb{E}_u[u_{k}]} )\sum\nolimits_{t = 0}^{k - 1} {{\eta _t}} {g_k} =\mathop {\lim }\limits_{k \to \infty } {{\eta _k}}\mathbb{E}_u[u_k](1-\frac{Var_u[u_k]}{\mathbb{E}_u[u_k]})g_k= 0.
\end{aligned}
\end{equation}
	Thus, $ \mathop {\lim }\limits_{k \to \infty }\mathbb{E}_u[u_{k}](1-\frac{Var_u[u_{k}]}{\mathbb{E}_u[u_{k}]} )\sum\nolimits_{t = 0}^{k - 1} {{\eta _t}} {g_k} = 0 $ from Assumption \ref{cond:condition_for_stepsizes}, i.e., from $ \sum\nolimits_{k=0}^{} {\frac{{{\eta _k}}}{{\sum\nolimits_{t = 0}^{k - 1} {{\eta _t}} }}} = \infty $. Therefore:
	\begin{equation}
	{g_k} = o\left( {\frac{1}{{\mathbb{E}_u[u_{k}](1-\frac{Var_u[u_{k}]}{\mathbb{E}_u[u_{k}]} )\sum\nolimits_{t = 0}^{k - 1} {{\eta _t}} }}} \right) \quad a.s.
	\end{equation}
	that is:
	\begin{equation}
	{g_k} = o\left( {\frac{1}{(\mathbb{E}_u[u_{k}]-Var_u[u_{k}]){\sum\nolimits_{t = 0}^{k - 1} {{\eta _t}} }}} \right) \quad a.s.
	\end{equation}
	Moreover, from (\ref{eq:g_k=grad_case1a}) it is that $ g_k \ge \mathop {\min }\limits_{t = 0, \ldots ,k - 1} {\left\| {\nabla f({x_t})} \right\|^2} \ge 0 $, which yields:
	\begin{equation}
	\mathop {\min }\limits_{t = 0, \ldots ,k - 1} {\left\| {\nabla f({x_t})} \right\|^2} = o\left( {\frac{(\mathbb{E}_u[u_{k}]-Var_u[u_{k}])^{-1}}{{\sum\nolimits_{t = 0}^{k - 1} {{\eta _t}} }}} \right),
	\end{equation}
	almost surely.
	\subsection{Case 1b: $ \mathbb{E}_u[u_{k+1}] \le \mathbb{E}_u[u_k] $ and $ \eta_k \le \frac{1}{LBc_2}  $.}

	Using $ u_k \le c_2 $ in (\ref{eq:following the steps in that paper:}):
	\begin{equation}
	\begin{aligned}
		&\mathbb{E}_k[f(x_{k+1} ) - f_*] + \frac{u_k \eta_{k}}{2} (2-c_2LB\eta_k) \left\| \nabla f (x_k) \right\|^2   \le (1+u_k^2 \eta_{k}^2 AL)(f(x_k)-f_*) + \frac{u_k^2\eta_k^2LC}{2}.
	\end{aligned}
	\end{equation}
	Using $ \eta_{k} \le 1 /(LBc_2) $ results in:
	\begin{equation}
	\begin{aligned}
		&\mathbb{E}_k[f(x_{k+1} ) - f_*] + \frac{u_k \eta_{k}}{2} \left\| \nabla f (x_k) \right\|^2  \le (1+u_k^2 \eta_{k}^2 AL)(f(x_k)-f_*) + \frac{u_k^2\eta_k^2LC}{2}.
	\end{aligned}
	\end{equation}
	By taking expectations with respect to $ u $:
	\begin{equation}\label{eq:some_eq_in_case1.1b}
	\begin{aligned}
	&\mathbb{E}_k[f(x_{k+1} ) - f_*] + \frac{ \mathbb{E}_u[u_k]\eta_{k}}{2} \left\| \nabla f (x_k) \right\|^2 \\&\le (1+\mathbb{E}_u[u_k^2] \eta_{k}^2 AL)(f(x_k)-f_*) + \frac{\mathbb{E}_u[u_k^2]\eta_k^2LC}{2}.
	\end{aligned}
	\end{equation}
	From Assumption \ref{cond:condition_for_stepsizes} it is that $ \sum\nolimits_{k = 0}^\infty  \frac{{u_k^2\eta _k^2LC}}{2} <  \sum\nolimits_{k = 0}^\infty  \frac{{c_2^2\eta _k^2LC}}{2}  < \infty  $. Also it is that $ \sum\nolimits_{k = 0}^\infty  {\eta _k^2} < \infty  $, which means $ \prod\nolimits_{k = 0}^\infty  {(1 + {\mathbb{E}_u}[u_k^2]\eta _k^2AL)} <  \prod\nolimits_{k = 0}^\infty  {(1 + c_2^2\eta _k^2AL)} < \infty  $. Thus from Lemma \ref{lemma:2.1}, $ f(x_k) - f_* $ converges a.s.
	Using (\ref{eq:recurrence_nonconvex_2}) to replace $ {\left\| {\nabla f({x_k})} \right\|^2} $ in (\ref{eq:some_eq_in_case1.1b}) gives:
	\begin{equation}
\begin{aligned}
& \mathbb{E}_k[f(x_{k+1}) - f_*] + \mathbb{E}_u[u_{k }] \frac{\sum\nolimits_{t=0}^{k} {\eta_t}}{2} g_{k+1} + \frac{\eta_k}{2} \mathbb{E}_u[u_k] g_k\\ &  \le (1+\mathbb{E}_u[u_k^2] \eta_{k}^2 AL)(f(x_k)-f_*)
+  \mathbb{E}_u[u_k] \frac{\sum\nolimits_{t=0}^{k-1} {\eta_t}}{2} g_{k} + \frac{\mathbb{E}_u[u_k^2]\eta_k^2LC}{2}.
\end{aligned}
\end{equation}
	Since $ \mathbb{E}_u[u_k] \ge \mathbb{E}_u[u_{k+1}] $:
	\begin{equation}
\begin{aligned}
& \mathbb{E}_k[f(x_{k+1}) - f_*] + \mathbb{E}_u[u_{k+1 }] \frac{\sum\nolimits_{t=0}^{k} {\eta_t}}{2} g_{k+1} + \frac{\eta_k}{2} \mathbb{E}_u[u_k] g_k\\ &  \le (1+\mathbb{E}_u[u_k^2] \eta_{k}^2 AL)(f(x_k)-f_*)
+  \mathbb{E}_u[u_k](1-\frac{Var_u[u_k]}{\mathbb{E}_u[u_k]^2}) \frac{\sum\nolimits_{t=0}^{k-1} {\eta_t}}{2} g_{k} + \frac{\mathbb{E}_u[u_k^2]\eta_k^2LC}{2}
\end{aligned}
\end{equation}
	and since $(1+\mathbb{E}_u[u_k^2] \eta_{k}^2 AL) > 1 $:
	\begin{equation}
\begin{aligned}
& \mathbb{E}_k[f(x_{k+1}) - f_* + \varphi_2(k+1)]  + \frac{\eta_k}{2} \mathbb{E}_u[u_k] g_k   \\& \le (1+\mathbb{E}_u[u_k^2] \eta_{k}^2 AL) \left( (f(x_k)-f_*) + \varphi_2  \right) + \frac{\mathbb{E}_u[u_k^2]\eta_k^2LC}{2}
\end{aligned}
\end{equation}
	where $  \varphi_2(k) =  \mathbb{E}_u[u_{k }] \frac{\sum\nolimits_{t=0}^{k-1} {\eta_t}}{2} g_{k} $. Then from $ u_k\le c_2 $ which gives $ \mathbb{E}_u[u^2_k] \le c^2_2 $ and from Assumption \ref{cond:condition_for_stepsizes} it is that $ \sum\nolimits_{k = 0}^\infty  {{\mathbb{E}_u}[u_k^2]\eta _k^2AL}  < {c_2^2}AL\sum\nolimits_{k = 0}^\infty  {\eta _k^2}  < \infty  $, which means $ \prod\nolimits_{k = 0}^\infty  {(1 + {\mathbb{E}_u}[u_k^2]\eta _k^2AL)}  < \infty  $. Moreover, it is that  $\frac{\mathbb{E}_u[u_k^2]LC}{2}\sum\nolimits_{k = 0}^\infty \eta_k^2 < \frac{c_2^2LC}{2}\sum\nolimits_{k = 0}^\infty \eta_k^2 < \infty $ from Assumption \ref{cond:condition_for_stepsizes}. Then, from Lemma \ref{lemma:2.1} this means that $  \sum\nolimits_{t = 0}^{k - 1} {{\eta _t}}\mathbb{E}_u[u_t]g_t < \infty  $ a.s. and also that $ ((f(x_{k}) - f_*) +  \mathbb{E}_u[u_{k}]\frac{\sum\nolimits_{t=0}^{k} {\eta_t}}{2} g_{k})  $ converges a.s.
	This gives that $ ({\varphi_2(k)})_k $ converges a.s. since it is also that $ (f(x_{k+1}) - f_*) $ converges a.s. This yields for $ \mathop {\lim }\limits_{k \to \infty } \varphi_2(k) $: 
	$$ \mathop {\lim }\limits_{k \to \infty } \frac{{{\eta _k}}}{{\sum\nolimits_{t = 0}^{k - 1} {{\eta _t}} }}\mathbb{E}_u[u_{k}]\sum\nolimits_{t = 0}^{k - 1} {{\eta _t}} {g_k} =\mathop {\lim }\limits_{k \to \infty } {{\eta _k}}\mathbb{E}_u[u_k]g_k = 0 .$$
	This means that  $ \mathop {\lim }\limits_{k \to \infty }\mathbb{E}_u[u_{k}]\sum\nolimits_{t = 0}^{k - 1} {{\eta _t}} {g_k} = 0 $ from Assumption \ref{cond:condition_for_stepsizes}, i.e., from $ \sum\nolimits_{k=0}^{} {\frac{{{\eta _k}}}{{\sum\nolimits_{t = 0}^{k - 1} {{\eta _t}} }}} = \infty $. Therefore:
	\begin{equation}
	{g_k} = o\left( {\frac{1}{{\mathbb{E}_u[u_{k}]\sum\nolimits_{t = 0}^{k - 1} {{\eta _t}} }}} \right) \quad a.s.
	\end{equation}
	Moreover, from (\ref{eq:g_k=grad_case1a}) it is that $ g_k \ge \mathop {\min }\limits_{t = 0, \ldots ,k - 1} {\left\| {\nabla f({x_t})} \right\|^2} \ge 0 $, which yields:
	\begin{equation}
	\mathop {\min }\limits_{t = 0, \ldots ,k - 1} {\left\| {\nabla f({x_t})} \right\|^2} = o\left( {\frac{1}{\mathbb{E}_u[u_{k}]{\sum\nolimits_{t = 0}^{k - 1} {{\eta _t}} }}} \right) \quad a.s.
	\end{equation}	

	\subsection{Case 2: $ \mathbb{E}_u[u_k] \le \mathbb{E}_u[u_{k+1}] $ and $ \eta_k \le \frac{1}{BL} $.}

	Using $ \eta_{k} \le 1 /(BL) $ in (\ref{eq:following the steps in that paper:}):
\begin{equation} \label{eq:eq_another_for_case1.22}
\begin{aligned}
&\mathbb{E}_k[f(x_{k+1}) - f_*] + \frac{\eta_{k}}{2} (2u_k-u^2_k) \left\| \nabla f (x_k) \right\|^2  \le (1+u_k^2 \eta_{k}^2 AL)(f(x_k)-f_*) + \frac{u_k^2\eta_k^2LC}{2}.
\end{aligned}
\end{equation}	
From $ u_k\le c_2 $ and Assumption \ref{cond:condition_for_stepsizes} it is that $ \sum\nolimits_{k = 0}^\infty  \frac{{u_k^2\eta _k^2LC}}{2} <  \sum\nolimits_{k = 0}^\infty  \frac{{c_2^2\eta _k^2LC}}{2}  < \infty  $. Also it is that $ \sum\nolimits_{k = 0}^\infty  {\eta _k^2} < \infty  $, which means $ \prod\nolimits_{k = 0}^\infty  {(1 + {\mathbb{E}_u}[u_k^2]\eta _k^2AL)} <  \prod\nolimits_{k = 0}^\infty  {(1 + c_2^2\eta _k^2AL)} < \infty  $. Thus, from Lemma \ref{lemma:2.1}, $ f(x_k) - f_* $ converges a.s. Taking expectations with respect to $ u $ yields:
\begin{equation}
\begin{aligned}
&\mathbb{E}_k[f(x_{k+1} ) - f_*] + \frac{ \eta_{k}}{2} (2\mathbb{E}_u[u_k]- \mathbb{E}_u[u^2_k]) \left\| \nabla f (x_k) \right\|^2 \\&\le (1+\mathbb{E}_u[u_k^2] \eta_{k}^2 AL)(f(x_k)-f_*) + \frac{\mathbb{E}_u[u_k^2]\eta_k^2LC}{2}.
\end{aligned}
\end{equation}
Using $ \mathbb{E}_u[u^2_k] = Var_u[u_k] + \mathbb{E}_u[u_k]^2 $:
\begin{equation}
\begin{aligned}
&\mathbb{E}_k[f(x_{k+1} ) - f_*] + \mathbb{E}_u[u_k]\frac{ \eta_{k}}{2} (2-\frac{\mathbb{E}_u^2[u_k]+Var_u[u_k]}{\mathbb{E}_u[u_k]}) \left\| \nabla f (x_k) \right\|^2 \\&\le (1+\mathbb{E}_u[u_k^2] \eta_{k}^2 AL)(f(x_k)-f_*) + \frac{\mathbb{E}_u[u_k^2]\eta_k^2LC}{2}.
\end{aligned}
\end{equation}
This means:
\begin{equation}
\begin{aligned}
&\mathbb{E}_k[f(x_{k+1} ) - f_*] + \mathbb{E}_u[u_k]\frac{ \eta_{k}}{2} (2-\mathbb{E}_u[u_k]-\frac{Var_u[u_k]}{\mathbb{E}_u[u_k]}) \left\| \nabla f (x_k) \right\|^2 \\& \le (1+\mathbb{E}_u[u_k^2] \eta_{k}^2 AL)(f(x_k)-f_*)+ \frac{\mathbb{E}_u[u_k^2]\eta_k^2LC}{2}.
\end{aligned}
\end{equation}
Since $ \mathbb{E}_u[u_k]<1 $ from the assumptions of the theorem:
\begin{equation}
\label{eq:eq_for_case1.22}
\begin{aligned}
&\mathbb{E}_k[f(x_{k+1} ) - f_*] + \mathbb{E}_u[u_k]\frac{ \eta_{k}}{2} (1-\frac{Var_u[u_k]}{\mathbb{E}_u[u_k]}) \left\| \nabla f (x_k) \right\|^2 \\& \le (1+\mathbb{E}_u[u_k^2] \eta_{k}^2 AL)(f(x_k)-f_*)+ \frac{\mathbb{E}_u[u_k^2]\eta_k^2LC}{2}.
\end{aligned}
\end{equation}
Then, using (\ref{eq:recurrence_nonconvex_2}) to replace $ {\left\| {\nabla f({x_k})} \right\|^2} $ in (\ref{eq:eq_for_case1.22}) yields (\ref{eq:eq_for_case1.2}).
Denoting $ \psi_k := \mathbb{E}_u[u_k](1-\frac{Var_u[u_k]}{\mathbb{E}_u[u_k]}) $, dividing (\ref{eq:eq_for_case1.2}) by $ \psi_k^2  $ and since $(1+\mathbb{E}_u[u_k^2] \eta_{k}^2 AL) > 1 $:
\begin{equation}
\begin{aligned}
& \mathbb{E}_k[\frac{1}{ \psi^2_k }(f(x_{k+1}) - f_*)] + \frac{1}{ \psi_k } \frac{\sum\nolimits_{t=0}^{k} {\eta_t}}{2} g_{k+1} + \frac{\eta_k}{2} \frac{g_k}{ \psi_k } 
\\ &\le  
\frac{(1+\mathbb{E}_u[u_k^2] \eta_{k}^2 AL)}{ \psi_k^2 }(f(x_k)-f_*) +
\frac{(1+\mathbb{E}_u[u_k^2] \eta_{k}^2 AL)}{ \psi_k } \frac{\sum\nolimits_{t=0}^{k-1} {\eta_t}}{2} g_{k} + \frac{\mathbb{E}_u[u_k^2]\eta_k^2LC}{2 \psi_k^2 }.
\end{aligned}
\end{equation}
Since $ \mathbb{E}_u[u_k] \le \mathbb{E}_u[u_{k+1}] $ which means $ \frac{1}{\mathbb{E}_u[u_k]} \ge \frac{1}{\mathbb{E}_u[u_{k+1}]} $
and since $ Var_u[u_{k+1}] \le Var_u[u_{k}]$	it is that:
\begin{equation}\label{eq:decreasing Varu[u_k]/Eu[u_k]}
\frac{1}{\mathbb{E}_u[u_k](1-\frac{Var_u[u_k]}{\mathbb{E}_u[u_k]})} \ge \frac{1}{\mathbb{E}_u[u_{k+1}](1-\frac{Var_u[u_{k+1}]}{\mathbb{E}_u[u_{k+1}]})}
\end{equation}
which means $ (\psi_{k})^{-1} \ge (\psi_{k+1})^{-1} $. This yields:
\begin{equation}
\begin{aligned}
& \mathbb{E}_k[\frac{1}{\psi_{k+1}}(f(x_{k+1}) - f_*) + \frac{1}{ \psi_{k+1} } \frac{\sum\nolimits_{t=0}^{k} {\eta_t}}{2} g_{k+1}] + \frac{\eta_kg_k}{2\psi_k }\\
&\le (1+\mathbb{E}_u[u_k^2] \eta_{k}^2 AL)\psi(k) + \eta_k^2\frac{c_2^2LC}{2N}
\end{aligned}
\end{equation}
	where $ \psi_k \ge \frac{c_2^2}{\mu_1^2} (\mu_1 + \mu_1^2 - c_2^2)^2 :=N $ for all $ k $, from using $ Var_u[u_{k}]= \mathbb{E}_u[u^2_{k}] -\mathbb{E}_u[u_{k}^2] $, then using $ \mathbb{E}_u[u^2_k]\le c_2^2 $ and $ \mathbb{E}_u[u_k] \le \mu_1 $, and doing some algebra.
	Then, from $ \mathbb{E}_u[u_k^2] \le c_2^2 $ and from Assumption \ref{cond:condition_for_stepsizes} it is that $ \sum\nolimits_{k = 0}^\infty  {{\mathbb{E}_u}[u_k^2]\eta _k^2AL}  < {c^2}AL\sum\nolimits_{k = 0}^\infty  {\eta _k^2}  < \infty  $, which means $ \prod\nolimits_{k = 0}^\infty  {(1 + {\mathbb{E}_u}[u_k^2]\eta _k^2AL)}  < \infty  $. Moreover, it is that  $\frac{c_2^2LC}{2N}\sum\nolimits_{k = 0}^\infty \eta_k^2 < \infty $ from Assumption \ref{cond:condition_for_stepsizes}. Then, from Lemma \ref{lemma:2.1} this means that 
	$  \sum\nolimits_{t = 0}^{k - 1}\frac{ {{\eta _t}}g_t}{ \psi_t } < \infty  $
	 a.s. and:
	 \begin{equation}\label{eq:close_to_final_result_for_case1.2}
	 	 \left(  \frac{1}{\psi^2_k}(f(x_{k}) - f_*) + \frac{1}{ \psi_k } \frac{\sum\nolimits_{t=0}^{k} {\eta_t}}{2} g_{k}  \right) \text{converges a.s. }
	 \end{equation}
	Note that from (\ref{eq:eq_another_for_case1.2}), by dividing by $ \psi_k^2  $ and using (\ref{eq:decreasing Varu[u_k]/Eu[u_k]}):
	\begin{equation}
\begin{aligned}
& \mathbb{E}_k[\frac{1}{ \psi_{k+1}^2 }(f(x_{k+1}) - f_*)]  + \frac{1}{ \psi_k }  \left\| \nabla f(x_k) \right\|^2  \\
&\le (1+\mathbb{E}_u[u_k^2] \eta_{k}^2 AL)  \frac{1}{ \psi_k^2} (f(x_k)-f_*)  + \frac{\mathbb{E}_u[u_k^2]\eta_k^2LC}{2 \psi_k^2 }.
\end{aligned}
\end{equation}
	From $ \mathbb{E}_u[u_k^2]\le c_2^2 $ and Assumption \ref{cond:condition_for_stepsizes} it is that $ \sum\nolimits_{t = 0}^{k - 1}\frac{\mathbb{E}_u[u_k^2]\eta_k^2LC}{2{\psi_k^2}} \le \frac{c_2^2LC}{2N} \sum\nolimits_{t = 0}^{k - 1}\eta_k^2 <  \infty $. Also, from Assumption \ref{cond:condition_for_stepsizes} it is that $ \sum\nolimits_{k = 0}^\infty  {{\mathbb{E}_u}[u_k^2]\eta _k^2AL}  < {c^2}AL\sum\nolimits_{k = 0}^\infty  {\eta _k^2}  < \infty  $ which gives $ \prod\nolimits_{k = 0}^\infty  {(1 + {\mathbb{E}_u}[u_k^2]\eta _k^2AL)}  < \infty  $. Then from Lemma \ref{lemma:2.1} it is that $ \frac{1}{\psi_k^2}(f(x_{k}) - f_*) $ converges a.s.	This means from (\ref{eq:close_to_final_result_for_case1.2}) that  $ {\left( {\left( \frac{1}{ \psi_k }{ \frac{\sum\nolimits_{t = 0}^{k - 1}{{\eta _t}}}{2} } \right){g_k}} \right)_k} $ converges a.s. 
	This yields:
	$$ \mathop {\lim }\limits_{k \to \infty } \frac{{{\eta _k}}}{{\sum\nolimits_{t = 0}^{k - 1} {{\eta _t}} }}\frac{1}{\psi_k }\sum\nolimits_{t = 0}^{k - 1} {{\eta _t}} {g_k} =\mathop {\lim }\limits_{k \to \infty }  \frac{\eta _kg_k}{ \psi_k } = 0 .$$ 
	This means that $ \mathop {\lim }\limits_{k \to \infty } \frac{1}{\psi_k}\sum\nolimits_{t = 0}^{k - 1} {{\eta _t}} {g_k} = 0 $ from Assumption \ref{cond:condition_for_stepsizes}, i.e., from $ \sum\nolimits_{k=0}^{} {\frac{{{\eta _k}}}{{\sum\nolimits_{t = 0}^{k - 1} {{\eta _t}} }}} = \infty $. Therefore, using the definition of $ \psi_k $ it is that:
	\begin{equation}
	{g_k} = o\left( {\frac{\mathbb{E}_u[u_{k}](1-\frac{Var_u[u_{k}]}{\mathbb{E}_u[u_{k}]} )}{{\sum\nolimits_{t = 0}^{k - 1} {{\eta _t}} }}} \right) \quad a.s.
	\end{equation}
	Moreover,  from (\ref{eq:g_k=grad_case1a}) it is that $ g_k \ge \mathop {\min }\limits_{t = 0, \ldots ,k - 1} {\left\| {\nabla f({x_t})} \right\|^2} \ge 0 $, which yields:
	\begin{equation}
	\mathop {\min }\limits_{t = 0, \ldots ,k - 1} {\left\| {\nabla f({x_t})} \right\|^2} = o\left( {\frac{\mathbb{E}_u[u_{k}](1-\frac{Var_u[u_{k}]}{\mathbb{E}_u[u_{k}]} )}{{\sum\nolimits_{t = 0}^{k - 1} {{\eta _t}} }}} \right) \quad a.s.
	\end{equation}
\end{proof}
\section{Discussion on the UMSLR Stochasticity Factor}
\label{sec:On_the_SF}
\label{sec:Discussion on the Stochasticity Factor}
\begin{figure}[b]
	\captionsetup[subfigure]{aboveskip=-1pt,belowskip=-5pt}
	\centering
	\begin{subfigure}[b]{0.49\linewidth}
		\centering
		\includegraphics[height=0.21\textheight,width=\textwidth]{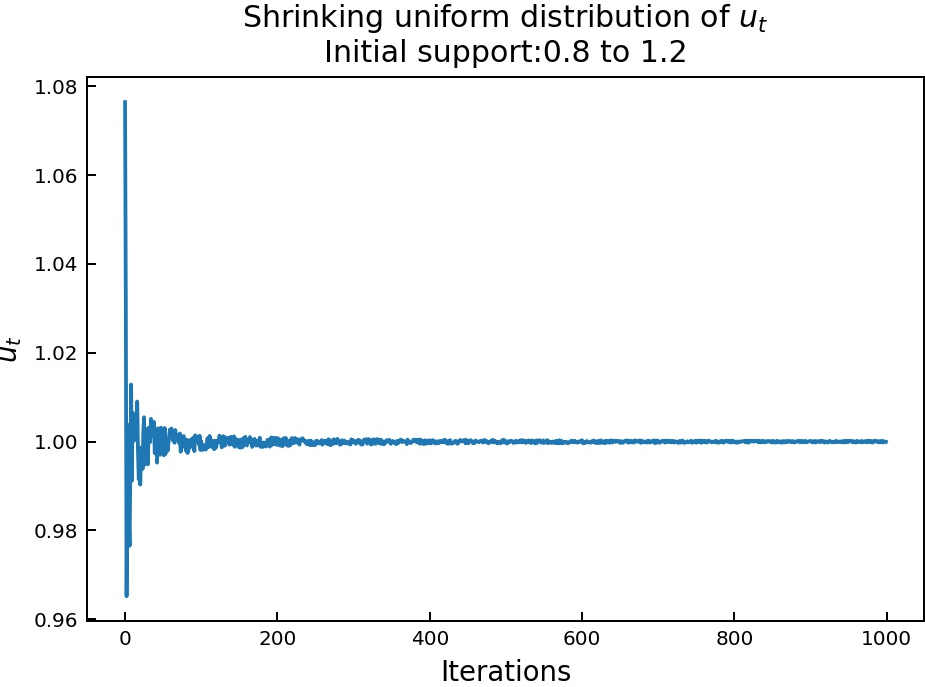}
		\label{fig:lr_0_8}
	\end{subfigure}
	\begin{subfigure}[b]{0.49\linewidth}
		\centering
		\includegraphics[height=0.21\textheight,width=\textwidth]{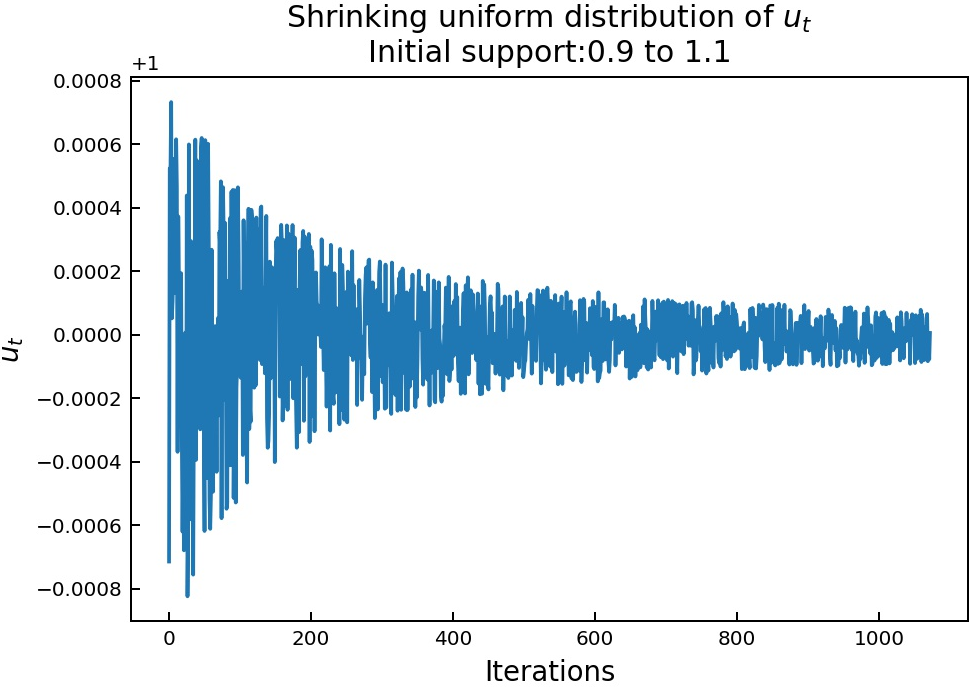}
		\label{fig:lr_0_9}
	\end{subfigure}
	\caption{Plots that illustrate two SFs for the  UMSLR scheme considered. Plot (a) shows an SF with support  $ [0.8^{1/k},1.2^{1/k}] $  and plot (b) a close-up of an SF with support $ [0.9^{1/k},1.1^{1/k}] $, where $ k $ is the iteration count.}
	\label{fig:lr}
\end{figure}
As mentioned in Section \ref{sec:Prelim_and_Assumptions}, SF distributions for which Assumption \ref{cond:condition_for_stochasticfactor} holds are able to be constructed. One such SF distribution is the uniform distribution and in particular when the SF is taken to follow $ u_k = \mathcal{U}\left(\sqrt[k+1]{c_1},\sqrt[k+1]{c_2}\right) $ with $ c_1 > 0 $. This  SF distribution will be used for the empirical results of this paper. The SF for some hyperparameter values is shown in Fig. \ref{fig:lr}.
 An MSLR scheme which uses a uniformly distributed SF will be referred to as UMSLR. In this work, UMSLR will refer to the aforementioned $ u_k $.
For the UMSLR scheme Assumption \ref{cond:condition_for_stochasticfactor} holds for a wide range of the hyperparameters $ c_1,c_2 $ rendering this assumption nonrestrictive in practice, as Proposition \ref{prop:Eu[uk]_monotone} shows.
%
%
\begin{proposition} \label{prop:Eu[uk]_monotone}
	Let $ u_k \sim \mathcal{U}\left(\sqrt[k+1]{c_1},\sqrt[k+1]{c_2}\right) $ with $ 0 < c_1 < c_2 $. Then $ \mathbb{E}_u[u_k] $ is monotone for all $ k\ge 0 $ for either (a): $ c_1\ge1$ and $c_2>1 $, or (b): $ c_1<1$ and $c_2\le1 $, or (c): 
	$ 0<c_1<1$ and  $c_2=e^{W(-c_1 lnc_1)} $
	, or (d): $ 0<c_1<1$ and $ \frac{1}{c_1} < c_2 $, where $ W $ denotes the principal branch of the Lambert-W function.
\end{proposition}
\begin{proof}
	The expectation of UMSLR with respect to $ u $ is $ \mathbb{E}_u[u_k] = \frac{\sqrt[k+1]{c_1}+\sqrt[k+1]{c_2}}{2}  $. For ease of notation, we denote $ k+1 $ as $ \bar{k} $. Then the derivative is:
	\begin{equation}
	({\mathbb{E}_u}[{u_k}])' =  - \frac{{\sqrt[\bar{k}]{{{c_1}}}\ln ({c_1}) + \sqrt[\bar{k}]{{{c_2}}}\ln ({c_2})}}{{{\bar{k}^2}}}
	\end{equation}

$Case \,\, (a):  c_1\ge1,c_2\ge1. $ 

If $ c_1 > 1 $ and $ c_2\ge1 $ then $ \mathbb{E}_u[u_k] $ is increasing since if $ \bar{k}_1 \le \bar{k}_2 $ then $  \mathbb{E}_u[u_{\bar{k}_1}] \le \mathbb{E}_u[u_{\bar{k}_2}] $. This is because $ \sqrt[\bar{k}_1]{c_1}  < \sqrt[\bar{k}_2]{c_1} $ and $ \sqrt[\bar{k}_1]{c_2} \le \sqrt[\bar{k}_2]{c_2} $ therefore $ \frac{\sqrt[\bar{k}_1]{c_1}+\sqrt[\bar{k}_1]{c_2}}{2} < \frac{\sqrt[\bar{k}_2]{c_1}+\sqrt[\bar{k}_2]{c_2}}{2}$, or equivalently $  \mathbb{E}_u[u_{\bar{k}_1}] < \mathbb{E}_u[u_{\bar{k}_2}] $.

$Case \,\, (b): 0<c_1 < 1,  c_2\le1. $

If $ c_1 < 1 $ and $ c_2 \le 1 $ then $ \mathbb{E}_u[u_k] $ is decreasing since if $ \bar{k}_1 \le \bar{k}_2 $ then $  \mathbb{E}_u[u_{\bar{k}_1}] \ge \mathbb{E}_u[u_{\bar{k}_2}] $. This is because in this case $ \sqrt[\bar{k}_1]{c_1} \ge \sqrt[\bar{k}_2]{c_1} $ and $ \sqrt[\bar{k}_1]{c_2} > \sqrt[\bar{k}_2]{c_2} $ therefore $ \frac{\sqrt[\bar{k}_1]{c_1}+\sqrt[\bar{k}_1]{c_2}}{2} > \frac{\sqrt[\bar{k}_2]{c_1}+\sqrt[\bar{k}_2]{c_2}}{2}$, or equivalently $  \mathbb{E}_u[u_{\bar{k}_1}] > \mathbb{E}_u[u_{\bar{k}_2}] $. 

$Case \,\, (c):  0<c_1<1,c_2=e^{W(-c_1lnc_1)}.  $

If $ c_2>1 $ and $ 0 < c_1 < 1 $ then the derivative of $ \mathbb{E}_u[u_k] $ has a single root at:
\[\underbar{k} =  - \frac{{\ln ({c_2}) - \ln ({c_1})}}{{\ln \left( { - \frac{{\ln ({c_2})}}{{\ln ({c_1})}}} \right)}}-1, \quad \underbar{k} =0, 1, \cdots \]
If $ \underbar{k}=0 $ this means that $ \mathbb{E}_u[u_k] $ will be monotone for all $ k \ge 0 $, i.e., it will be monotone for the whole duration of the algorithm. The condition $ \underbar{k}=0 $ is satisfied when $ \frac{{\ln ({c_2}) - \ln ({c_1})}}{{\ln \left( { - \frac{{\ln ({c_2})}}{{\ln ({c_1})}}} \right)}} -1=0 $, or equivalently when $ c_1 \ln c_1 + c_2 \ln c_2 = 0 $ from which it follows $ e^{\ln c_2} \ln c_2 = -c_1 \ln c_1  $. This means that $ \ln c_2 = W(-c_1 \ln c_1) $. Then, by utilizing the principal branch of the Lambert-W function it is that $ c_2 = e^{W(-c_1 \ln c_1)} $. 

$Case \,\, (d): 0<c_1<1, \frac{1}{c_1} < c_2. $

Moreover, the same is true when $ \underbar{k}<0 $. This can happen if $ \underbar{k}\le -1 $, that is,
$  \frac{{\ln ({c_2}) - \ln ({c_1})}}{{\ln \left( { - \frac{{\ln ({c_2})}}{{\ln ({c_1})}}} \right)}} \ge 0 $.
This means that $ \ln ({c_2}) - \ln ({c_1}) >0 $, i.e., $ c_2>c_1 $ or equivalently $ c_2 >1>c_1 $, and $ \ln \left( { - \frac{{\ln ({c_2})}}{{\ln ({c_1})}}} \right) >0 $, i.e., $ \frac{1}{c_1}<c_2 $.
Therefore, $ \mathbb{E}_u[u_k] $ is monotone for either $ c_1\ge1,c_2\ge1 $ or $ c_1\le1,c_2\le1 $, or $ 0<c_1<1$ and  $1<c_2<e^{W(1/e)} $ or $ 0<c_1<1$ and $ \frac{1}{c_1} < c_2 $, and thus the proof.
\end{proof}

It is noted that the SF distribution of UMSLR satisfies the required assumptions of Theorem \ref{thm:3.2}, and in particular case 1.2 This is because the experiments use the hyperparameter values from case (b) of Proposition \ref{prop:Eu[uk]_monotone} i.e., $ 0<c_1<c_2<1 $. These values yield a monotone, and in fact increasing, $ \mathbb{E}_u[u_k] $. Moreover, for every finite values $ c_1 $ and $ c_2 $, $ \mathbb{E}_u[u_k] $ has an infimum. 

Furthermore, case 1.2 requires increasing $ \mathbb{E}_u[u_k] $, decreasing $ Var_u[u_k] $,  $ \mathbb{E}_u[u_k]<1 $ and  $ Var_u[u_k] < \mathbb{E}_u[u_k] $. The first requirement is satisfied as aforementioned. The third is satisfied since $ \mathbb{E}_u[u_k] =(c_1^{1/k}+c_2^{1/k})/2<1  $ for $ 0<c_1<c_2<1 $. The fourth since  $ Var_u[u_{k}] = 	\frac{(c_2^{1/k}-c_1^{1/k})^2}{12}	< \frac{(c_2^{1/k}+c_1^{1/k})^2 }{4} = \mathbb{E}_u[u_k]^2 < \mathbb{E}_u[u_k] $ where the first inequality follows from the expression of the variance of a uniformly distributed, at each $ k $, random variable, and the latter inequality follows because $ \mathbb{E}_u[u_k] < 1 $. The second is satisfied because for the variance derivative it holds that $ \frac{{\partial Va{r_u}[{u_k}]}}{{\partial {u_k}}} =  - \frac{1}{{6{k^2}}}(\sqrt[k]{c_2} - \sqrt[k]{c_1})(\sqrt[k]{c_2}\log (c_2) - \sqrt[k]{c_1}\log (c_1)) < 0 $ since for the $ c_1,c_2 $ values considered, both the parentheses factors are positive. The first factor is positive when $ c_2 > c_1 $ and the second when $ \sqrt[k]{c_2} \ln c_2 > \sqrt[k]{c_1} \ln c_1 $ or equivalently $ e^{\ln c_2} \ln c_2^{1/k} > e^{\ln c_1} \ln c_1^{1/k} $. Applying the principal branch of the Lambert W function then gives $ \ln c_2^{1/k} > W(c_1^{1/k} \ln c_1^{1/k} ) $, which means $ c_2 > e^{k W(c_1^{1/k} \frac{1}{k} \ln c_1 )} = c_1 $. Therefore, the variance is decreasing for the hyperparameters $ 0<c_1<c_2<1 $  used in UMSLR. 

Finally, it is verified that UMSLR  provides faster a.s. convergence rates for SGD than a determinisitc learning rate since from the above it is readily obtained that $ 0<\mathbb{E}_u[u_k](1-\frac{Var_u[u_k]}{\mathbb{E}_u[u_k]})<1 $  which means   $ \mathbb{E}_u[u_{k}]-Var_u[u_{k}] <1  $. This yields that the UMSLR a.s. convergence rate of  $ o\left( \frac{\mathbb{E}_u[u_{k}]-Var_u[u_{k}] }{{\sum\nolimits_{t = 0}^{k - 1} {{\eta _t}} }} \right)  $ is faster than the deterministic-case rate of $ o\left( \frac{1 }{{\sum\nolimits_{t = 0}^{k - 1} {{\eta _t}} }} \right)  $. This is demonstrated experimentally in the next section.

\section{Results}
\label{sec:results}

Experimental results using a stochastic learning rate scheme applied to SGD in the nonconvex and smooth setting and comparisons to the determinisitc-learning-rate SGD are presented in this section. In specific, the stochastic learning rate scheme used was UMSLR, i.e., MSLR with an SF described by $ u_k \sim \mathcal{U}\left(\sqrt[k+1]{c_1},\sqrt[k+1]{c_2} \right) $. The plots depict either training or testing running losses which were produced by varying the random seed 40 times and then taking the average of the results. In summary, when the algorithms where equipped with UMSLR they produced significantly accelerated convergence and overall optimization performance than when they employed a deterministic learning rate in the nonconvex setting, as predicted by Theorem \ref{thm:3.2}. The loss function used was the cross-entropy loss.

In order to compare the results of UMSLR SGD to its deterministic learning rate counterpart, t-tests were used for all of the plots. Bonferroni correction with a family-wise error rate of 0.05 was applied to these t-tests to determine which plots produced statistically significant results. The step size was taken to be a constant. The network architecture was ResNet18 and the datasets used were CIFAR-10 and CIFAR-100 (\cite{CIFARdata09}). For both algorithms the hyperparameters, i.e., learning rates, momentums and the UMSLR SF constants, were chosen via grid search. All experiments were conducted for 20 epochs each.  

As depicted in Figs. \ref{fig:SGD_CIFAR10} and \ref{fig:SGD_CIFAR100}, the UMSLR scheme demonstrates significantly enhanced optimization performance compared to a deterministic learning rate. Figures \ref{fig:SGD_CIFAR10_train} and \ref{fig:SGD_CIFAR100_train} demonstrate accelerated minimization performance for the training phase as predicted by Theorem \ref{thm:3.2}. Even though, given the nonconvex setting, the theorem concerns accelerated convergence to a local minimum, for the experiments conducted convergence was achieved to the global minimum, i.e, zero, as demonstrated by the plots. Furthermore, in \ref{fig:SGD_CIFAR10_test} and \ref{fig:SGD_CIFAR100_test} the UMSLR scheme exhibits enhanced testing-phase performance compared to the deterministic-learning-rate case since it can be observed that the algorithm with UMSLR achieves both a faster and a lower minimum testing loss  than their deterministic learning-rate versions. All of the aforementioned results were statistically significant using a family-wise error rate of 0.05. Even though the experiments were run for 20 epochs each, the plots for the testing losses for SGD depict the losses up to the point when overfitting begins, for clarity of presentation (i.e., up to the point when the testing losses reach their minimum). 

Furthermore, even if the purpose of this work is not producing state-of-the-art testing performance but rather to showcase the viability of stochastic-learning-rate methods, testing empirical results were included for completeness to demonstrate the improved testing performance when using a stochastic learning rate scheme. More specifically, in all of the cases presented in these plots the algorithm with an MSLR type of learning rate achieves both a faster and lower minimum testing loss than its deterministic-learning-rate counterpart, even if the theoretical analysis conducted in Theorem \ref{thm:3.2} concerned only the training phase and not the generalization performance of the MSLR scheme.  In-detail analysis, both theoretical and experimental, of the testing performance for stochastic-learning-rate algorithms is left as an avenue for future work.
\begin{figure}
	\centering
	\begin{subfigure}[b]{0.49\textwidth}
		\centering
		\includegraphics[width=\textwidth]{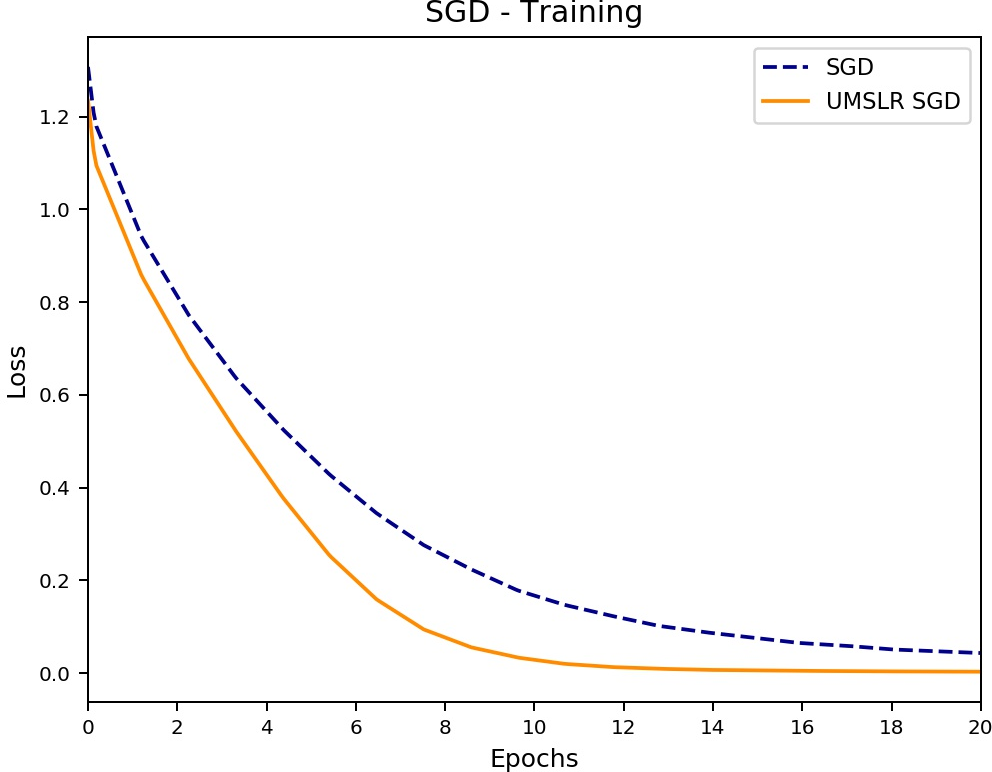}
		\caption{}
		\label{fig:SGD_CIFAR10_train}
	\end{subfigure}
	\begin{subfigure}[b]{0.49\textwidth}
		\centering
		\includegraphics[width=\textwidth]{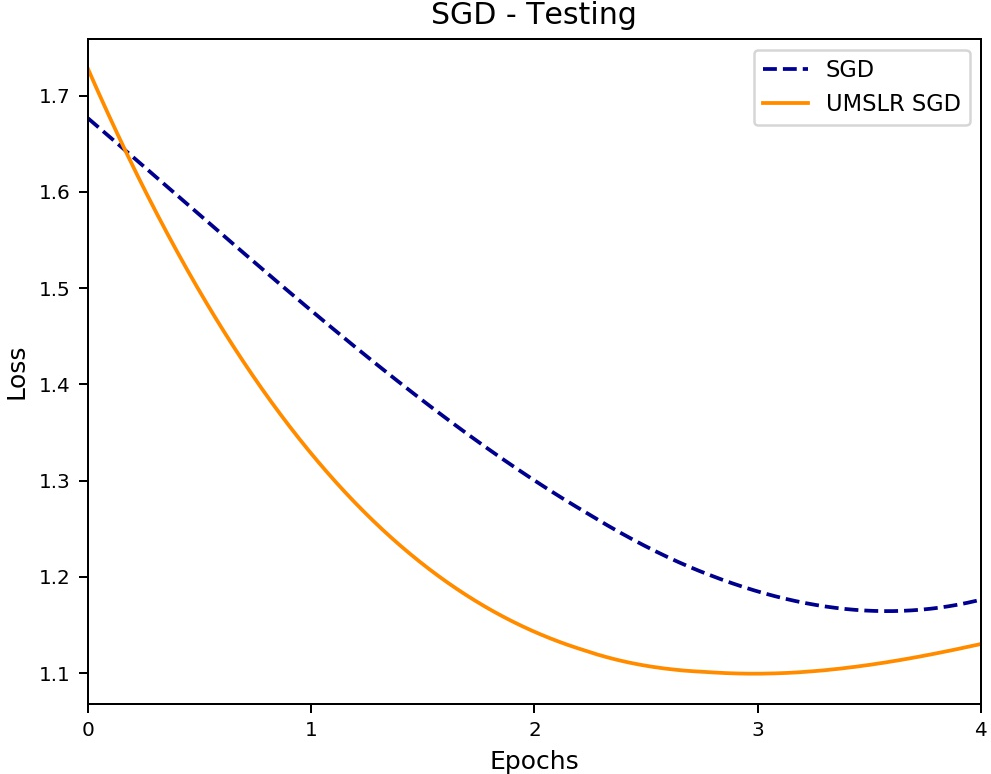}
		\caption{}
		\label{fig:SGD_CIFAR10_test}
	\end{subfigure}
	\caption{Plots that illustrate the UMSLR scheme (in orange) and the deterministic learning rate (in blue) schemes applied to SGD under the nonconvex setting for CIFAR-10 and hyperparameters  $ c_1=0.3 $ and $ c_2=0.8 $.
	}
	\label{fig:SGD_CIFAR10}
\end{figure}
\begin{figure}
	\centering
	\begin{subfigure}[b]{0.49\textwidth}
		\centering
		\includegraphics[width=\textwidth]{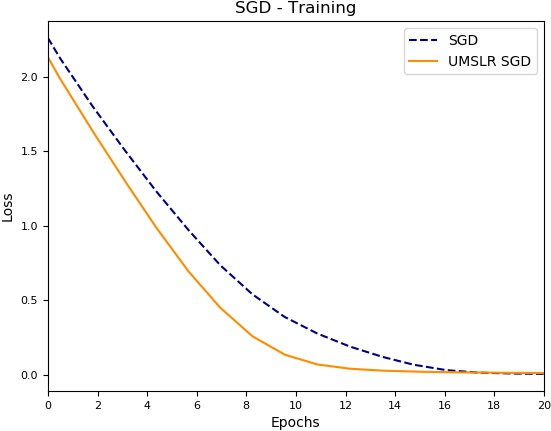}
		\caption{}
		\label{fig:SGD_CIFAR100_train}
	\end{subfigure}
	\begin{subfigure}[b]{0.49\textwidth}
		\centering
		\includegraphics[width=\textwidth]{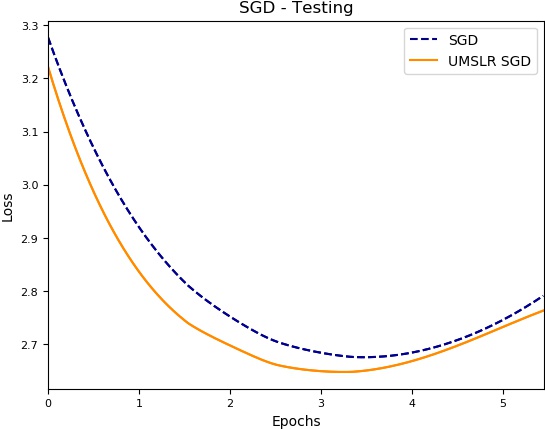}
		\caption{}
		\label{fig:SGD_CIFAR100_test}
	\end{subfigure}
	\caption{Plots that illustrate the UMSLR scheme (in orange) and the deterministic learning rate (in blue) schemes applied to SGD under the nonconvex setting for CIFAR-100 and hyperparameters  $ c_1=0.6 $ and $ c_2=0.9 $.}
	\label{fig:SGD_CIFAR100}
\end{figure}
\section{Conclusion and Future Work}
\label{sec:conclusion}
This work introduced a learning-rate scheme that makes the learning rate stochastic by multiplicatively equipping it with a random variable, coined stochastic factor. Almost sure convergence rates of the SGD algorithm in the nonconvex and smooth setting employing the stochastic learning rate scheme was theoretically analyzed. The theoretical discussion that followed suggested how the hyperparameter values introduced by the stochastic-learning-rate scheme should be chosen. Empirical results on popular datasets demonstrated noticeable increase in optimization performance, presenting stochastic-learning-rate schemes as a viable option for enhancing performance. Gains in testing performance were also observed.
In-depth generalization performance studies, convergence analysis of algorithms besides SHB and SGD, and investigating the effect of various stochastic factor distributions and hyperparameters on algorithm performance are some paths for future work.
\newpage
\bibliographystyle{abbrvnat}
\bibliography{refnew}

\end{document}